\documentclass{amsart}
\usepackage{amssymb,amsmath}
\usepackage{url}
\usepackage{graphics}
\usepackage{nicefrac}
\usepackage[usenames]{color}
\usepackage{hyperref}

\newcommand{\flip}[1]{\raisebox{\depth}{\rotatebox{180}{#1}}}
\newcommand{\board}[9]{
\begin{tabular}{|c|c|c|}\hline
#1 & #2 & #3\\\hline
#4 & #5 & #6\\\hline
#7 & #8 & #9\\\hline
\end{tabular}}
\definecolor{myred}{rgb}{0.75,0,0}
\definecolor{mygreen}{rgb}{0,0.75,0}
\definecolor{myblue}{rgb}{0,0,0.75}
\newcommand{\one}{{\color{myred} 1}}
\newcommand{\two}{{\color{myred} 2}}
\newcommand{\three}{{\color{myred} 3}}
\newcommand{\four}{{\color{mygreen} 4}}
\newcommand{\five}{{\color{mygreen} 5}}
\newcommand{\six}{{\color{mygreen} 6}}
\newcommand{\seven}{{\color{myblue} 7}}
\newcommand{\eight}{{\color{myblue} 8}}
\newcommand{\nine}{{\color{myblue} 9}}
\newcommand{\idboard}{\board{\one}{\two}{\three}{\four}{\five}{\six}{\seven}{\eight}{\nine}}
\newcommand{\spin}{{\rm Spin}}
\newcommand{\spinmn}{\spin_{m\times n}}

\newcommand{\sid}{\iota}
\newcommand{\bid}{{\bf 0}}

\newcommand{\be}{{\bf e}}
\newcommand{\bu}{{\bf u}}
\newcommand{\bv}{{\bf v}}

\newcommand{\RR}{\mathcal{R}}
\newcommand{\Rij}{\RR_{i\times j}}
\newcommand{\Rji}{\RR_{j\times i}}
\newcommand{\Z}{\mathbb{Z}}
\renewcommand{\SS}{\mathcal{S}}
\newcommand{\TT}{\mathcal{T}}
\newcommand{\UU}{\mathcal{U}}
\newcommand{\Sij}{\SS_{i\times j}}
\newcommand{\Sonetwo}{\SS_{1\times 2}}
\newcommand{\Soneone}{\SS_{1\times 1}}
\newcommand{\Stwotwo}{\SS_{2\times 2}}
\newcommand{\Sonethree}{\SS_{1\times 3}}
\newcommand{\Stwothree}{\SS_{2\times 3}}
\newcommand{\Sthreethree}{\SS_{3\times 3}}
\newcommand{\dist}{\rho}
\newcommand{\cay}{\operatorname{Cay}}
\newcommand{\wt}{\operatorname{wt}}
\newcommand{\projmap}{\pi}
\newtheorem{theorem}{Theorem}
\newtheorem{definition}{Definition}

\newtheorem{proposition}{Proposition}
\newtheorem{corollary}{Corollary}
\newtheorem{lemma}{Lemma}

\title[The Mathematics of Spinpossible]{The Mathematics of Spinpossible}
\author{Alex Sutherland and Andrew Sutherland}
\date{}
\begin{document}
\maketitle
\section{Introduction}
Spinpossible$^\text{\tiny\texttrademark}$ is played on $3\times 3$ board of tiles numbered from 1 to~9, each of which my be right-side-up or up-side-down.
One possible starting position is the board:
\medskip
\begin{center}
\board{\flip{\five}}{\flip{\four}}{\flip{\nine}}{\flip{\two}}{\flip{\one}}{\flip{\six}}{\seven}{\eight}{\flip{\three}}
\end{center}
\medskip
The objective of the game is to return the board to the standard configuration:
\medskip
\begin{center}
\idboard
\end{center}
\medskip
This is accomplished by a sequence of \emph{spins}, each of which rotates a rectangular region of the board by $180^\circ$.
The goal is to minimize the number of spins used.  The starting board above may be solved using two spins:
\medskip
\begin{center}
\board{\flip{5}}{\flip{4}}{\flip{\nine}}{\flip{2}}{\flip{1}}{\flip{\six}}{\seven}{\eight}{\flip{\three}}$\quad\Longrightarrow\quad$
\board{\one}{\two}{\flip{9}}{\four}{\five}{\flip{6}}{\seven}{\eight}{\flip{3}}$\quad\Longrightarrow\quad$\idboard
\end{center}
\medskip
In this example, the first spin rotates the the $2\times 2$ rectangle in the top left, and the second spin rotates the $3\times 1$ rectangle along the right edge.
You can play the game online at \url{http://spinpossible.com}

In these notes we give a mathematical description of this game, and some of its generalizations, and consider various questions that naturally arise.
Perhaps the most obvious is this: is it always possible to return a given board to the standard configuration with a sequence of spins?
We shall see shortly that the answer is yes.

A more difficult question is the following: what is the maximum number of spins required to solve any board?
An exhaustive search has found that 9 spins are always sufficient (and sometimes necessary), but no short proof of this fact is known.

\section{A mathematical description of the game}

We begin by defining the group $\spinmn$, for a fixed pair of positive integers $m$ and $n$  with product $N=mn$.
Let $S_N$ denote the symmetric group on~$N$ letters with the action on the right (so the permutation $\alpha\beta$ applies $\alpha$ and then~$\beta$), and 
let $V_N=(\Z/2\Z)^N$ denote the additive group of $N$-bit vectors.
For any vector $\bv=(v_1,\ldots,v_N)\in V_N$ and permutation $\alpha\in S_N$, we use $\bv^\alpha=(v_{\alpha^{-1}(1)},\ldots,v_{\alpha^{-1}(N)})$ to denote the vector obtained by applying $\alpha$ to $\bv$.

\begin{definition}
The group $\spinmn$ is the set $\{(\alpha,\bu)\colon \alpha\in S_{mn}, \bu\in V_{mn}\}$ under the operation $(\alpha,\bu)(\beta,\bv) = (\alpha\beta,\bu^\beta+\bv)$.
Equivalently, $\spinmn$ is the wreath product $\Z/2\Z\wr S_N$.
\end{definition}
\noindent

Readers familiar with Coxeter groups will recognize $\spinmn$ as the hyperoctahedral group of degree $N$ (the symmetry group of both the $N$-cube and the $N$-octahedron), equivalently, the Weyl group of type $B_N$ (and $C_N$).   This group can also be represented using signed permutation matrices, but the representation as a wreath product is better suited to our purposes here.  The definition of the group $\spinmn$ depends only on $N=mn$, however the integers $m$ and $n$ determine the set of generators we will be defining shortly, and these play a key role in the game (playing Spinpossible on a $1\times 9$ board would be much less interesting!).

A \emph{board} is an $m\times n$ array of uniquely identified elements called \emph{tiles}, which we number from $1$ to $mn$.
Each tile may be oriented positively (right-side-up), or negatively (upside-down).
The \emph{positions} of the board are fixed locations, which for convenience we regard as unit squares in the plane, also numbered from $1$ to $mn$, starting at the top left and proceeding left to right, top to bottom.
The \emph{standard board} has tile $i$ in position $i$, with positive orientation.

There is a 1-to-1 correspondence between $m\times n$ boards and elements of $\spinmn$, but we will generally think of elements of $\spinmn$ as acting on the set of all $m\times n$ boards as follows: the element $(\alpha,\bv)$ first permutes the tiles by moving the tile in position $i$ to position $\alpha(i)$, and then reverses the orientiation of the tile in the $i$th position if and only if $v_i=1$.
Of course this is just the action of $\spinmn$ on itself.

The projection map $\pi\colon\spinmn\to S_N$ that sends $(\alpha,\bu)$ to $\alpha$ is a group homomorphism, and we have the short exact sequence
\[
1\quad\longrightarrow\quad V_N\quad\longrightarrow\quad\spinmn\quad\longrightarrow\quad S_N\quad\longrightarrow\quad 1.
\]
It is worth emphasizing that the projection from $\spinmn$ to $V_N$ is \emph{not} a group homomorphism (for $N>1$).


We now distinguish the elements of $\spinmn$ that correspond to \emph{spins}, the moves permitted in the game.
A \emph{rectangle} $R$ specifies a rectangular subset of the positions on an $m\times n$ board, and has dimensions $i\times j$, with $1\le i\le m$ and $1\le j\le n$.
A spin rotates some rectangle by $180^\circ$.
It is reasonably clear what this means, but to make it more precise we define a notion of distance that will be useful later.

The \emph{distance} $\dist(p_1,p_2)$ between positions $p_1$ and $p_2$ is measured by applying the $\ell_1$-norm to the centers of the corresponding unit squares.
Two positions are \emph{adjacent} when they have a single edge in common, equivalently, when the distance between them is 1.
For a position $p$ and a rectangle $R$, we use $\dist(p,R)$ to denote the distance from the center of $p$ to the center of $R$ (again using the $\ell_1$-norm).

\begin{definition}
The \emph{spin about} $R$ is the element of $\spinmn$ that transposes the tiles in positions $p_1,p_2\in R$ if and only if $\dist(p_1,p_2)=2\dist(p_1,R)=2\dist(p_2,R)$, and then reverses the orientation of each tile in $R$.
\end{definition}

We say that an element of $\spinmn$ is a \emph{spin} if it is a spin about some rectangle~$R$.
The proposition below records some useful facts about spins.
The proofs are straight-forward, but for the sake of completeness we fill in the details.  They can (and probably should) be skipped on a first reading.

\begin{proposition}\label{prop:properties}
Let $s_1$ and $s_2$ be spins about rectangles $R_1$ and $R_2$ respectively.
\begin{enumerate}
\item
$s_1$ is its own inverse (as is $s_2$).
\item
$s_1s_2$ is not a spin.
\item
$s_1s_2=s_2s_1$ if and only if $R_1$ and $R_2$ are disjoint or have a common center.
\item
$s_1s_2s_1$ is a spin $s_3$ if and only if either $s_1$ and $s_2$ commute or $R_1$ contains~$R_2$.
The rectangle of $s_3$ has the same shape as $R_2$.
\end{enumerate}
\end{proposition}
\begin{proof}
(1) is clear.
For (2), suppose $s_3=s_1s_2$ is a spin about some rectangle $R_3$.  Then $s_1\ne s_2$, and therefore $R_1\ne R_2$, since the identity element is not a spin.
Now suppose there exist positions $p_1\in R_1-R_2$ and $p_2\in R_2-R_1$.  Then $s_3$ moves the tile in position $p_i$ to the same location that $s_i$ does, which implies that $R_i$ and $R_3$ have a common center, for $i=1,2$.
But then there is a position in $R_3$ containing this common center (either in its center or along an edge) and $s_3=s_1s_2$ does not change the orientation of the tile in this position, which is a contradiction.
Now assume without loss of generality that $R_1$ properly contains $R_2$.
Let $p_1$ and $p_2$ be corners of $R_1$ not contained in $R_2$ (let $p_1=p_2$ if $R_1$ has width or height 1, and $p_1\ne p_2$ otherwise).
Then $s_3$ acts on the tiles in positions $p_1$ and $p_2$ the same way that $s_1$ does, and this implies that $R_3$ and $R_1$ have a common center and that $R_1\subset R_3$.
But $R_3$ must lie in the union of $R_1$ and $R_2$, so $R_3=R_1$ and $s_3=s_1$, but then $s_2$ is the identity, which is again a contradiction.  So $s_1s_2$ is not a spin.

We now address (3).
For any position $p$ not in the interesection of $R_1$ and $R_2$, both $s_1s_2$ and $s_2s_1$ have the same effect on the tile $t$ in position $p$.
Now suppose $p$ is in the intersection of $R_1$ and $R_2$.
Then the orientation of $t$ is preserved by both $s_1s_2$ and $s_2s_1$, so we need only consider the position to which $t$ is moved.
The product of two rotations by $\pi$ is a translation (possibly trivial).
Reversing the order of the rotations yields the inverse translation, thus $t$ is moved to the same position if and only if the translation is trivial, which occurs precisely when $R_1$ and $R_2$ have a common center.
This proves (3).

For (4), it is clear that if $s_1$ and $s_2$ commute then $s_3=s_2$ is a spin.
Now suppose $R_1$ contains $R_2$.  Let $R_3$ be the inverse image of $R_2$ under the permutation $\projmap(s_1)$, and let $s_3=s_1s_2s_1$.
For tiles in $R_3$, the action of $s_3$ is the product of three rotations by $\pi$, which is again a rotation by $\pi$, and the center of this rotation is the center of $R_3$.
Thus $s_3 $is a spin about $R_3$, which has the same shape as $R_2$.

To prove the other direction of (4), suppose for the sake of contradiction that $s_3$ is a spin about some rectangle $R_3$, that $R_1$ and $R_2$ are not disjoint, do not have a common center, and that $R_1$ does not contain $R_2$.  These assumptions guarantee the existence of a position $p \in R_2-R_1$ whose image under $\projmap(s_2)$ is in~$R_1$.
The action of $s_3$ on the tile $t$ in position $p$ is the same as $s_2s_1$, which is two rotations by $\pi$, hence a translation.  But $R_1$ and $R_2$ do not have a common center, so this translation is non-trivial and $s_3$ moves tile $t$ without changing its orientation, contradicting our assumption that $s_3$ is a spin.
\end{proof}

Each spin is uniquely determined by its rectangle $R$, thus we may specify a spin in the form $[p_1,p_2]$, where $p_1$ and $p_2$ identify the positions of the upper left and lower right corners of $R$ (respectively).
For example, on a $3\times 3$ board the spin about the $2\times 2$ rectangle in the upper right corner is
\[
[2,6] = \bigl((2\ 6)(3\ 5),011011000\bigr).
\]
The moves permitted in a game of Spinpossible on an $m\times n$ board are precisely the set $\SS=\SS(m,n)$ of all spins $[p_1,p_2]$, where $1\le p_1\le p_2\le mn$
(we consider variations of the game that place restrictions on $\SS$ in \S\ref{section:restrictions}).

In mathematical terms, the game works as follows: given an element $b\in\spinmn$ (the starting board), write $b^{-1}$ as a product $s_1s_2s_3\cdots,s_k$ of elements in $\SS$, with~$k$ as small as possible (a \emph{solution}).
Applying the spins $s_1, s_2, \ldots, s_k$ to $b$ then yields the identity (the standard board).
In general there will be many solutions to $b$, but some boards have a unique solution; this topic is discussed further in \S\ref{section:unique}.

Let $\Rij$ denote the subset of $\spinmn$ that are spins about an $i\times j$ rectangle.
The set $\Rij$ is necessarily empty if $i > m$ or $j > n$, and we may have $\Rij = \emptyset$ even when $\Rji\ne\emptyset$ (although this can occur only when $m\ne n$).
In these notes we shall always consider the sets $\Rij$ and $\Rji$ together, thus we define $\Sij=\Rij\cup\Rji$.
The set of spins $\SS$ in $\spinmn$ is then the union of the sets $\Sij$, each of which we refer to as a \emph{spin type}.

\begin{proposition}\label{prop:cardinality}
Assume $1\le i \le m$, $1\le j\le n$, and $m\le n$.  The following hold:
\begin{enumerate}
\item
$|\Rij| = (m+1-i)(n+1-j)$.
\smallskip
\item
$
|\Sij| = \begin{cases}
|\Rij|+|\Rji|\qquad &\text{for } i\ne j,\\
|\Rij| &\text{for } i=j.\\
\end{cases}
$
\medskip
\item
$|\SS| = \binom{m+1}{2}\binom{n+1}{2}.$
\medskip
\item
There are $\frac{1}{2}m(2n-m+1)$ distinct spin types $\Sij$ in $\spinmn$.
\end{enumerate}
\end{proposition}
\begin{proof}
For (1), we note that there are $(m+1-i)(n+1-j)$ possible locations for the upper left corner of an $i\times j$ rectangle on an $m\times n$ board.
The formula in (2) is immediate.  For (3) we have
\[
\sum_{i=1}^m\sum_{j=1}^n(m+1-i)(n+1-j) = \sum_{i=1}^m\sum_{j=1}^n ij = \binom{m+1}{2}\binom{n+1}{2},
\]
and for (4) we have
\[
\sum_{i=1}^m\sum_{j=i}^n 1 = \sum_{i=1}^m (n+1-i) = m(n+1)-\binom{m+1}{2} = \frac{m(2n-m+1)}{2}.
\]
\end{proof}

It is well known that the symmetric group $S_N$ is generated by the set of all transpositions (permutations that swap two elements and leave the rest fixed).
Slightly less well known is the fact that $S_N$ is generated by any set of transpositions that form a connected graph, as described in the following lemma.

\begin{lemma}\label{lemma:transpositions}
Let $E\subseteq S_N$ be a set of transpositions $(v_i,v_j)$ acting on a set of vertices $V=\{v_1,\ldots,v_N\}$.
Let $G$ be the undirected graph on $V$ with edge set $E$.
Then $E$ generates $S_N$ if and only if $G$ is connected.
\end{lemma}
\begin{proof}
It suffices to show that $E$ generates every transposition in $S_N$.
If the sequence of edges $(e_1,\ldots,e_k)$ is a path from $v_i$ to $v_j$ in $G$, then the permutation
\[
e_1e_2\ldots e_{k-2}e_{k-1}e_ke_{k-1}e_{k-2}\ldots e_2e_1
\]
is the transposition $(v_i,v_j)$.
Let $H$ be the subgroup of $S_N$ generated by~$E$.
The $H$-orbits of $V$ correspond to connected components of $G$, and $H$ can achieve any permutation of the vertices in a given component, since it can transpose any pair of vertices connected by a path.
Thus $H=S_N$ if and only if $G$ is connected.
\end{proof}

\begin{corollary}\label{cor:generates}
$\Soneone \cup \Sonetwo$ generates $\spinmn$.
\end{corollary}
\begin{proof}
The set $\Sonetwo$ consists of transpositions that form a connected graph whose vertices are the positions on an $m\times n $ board with edges between adjacent positions.
If follows from Lemma \ref{lemma:transpositions} that, given any element $(\alpha,\bu)$ in $\spinmn$, there is a vector $\bv\in B_{mn}$ for which we can construct $(\alpha,\bv)$ as a product of elements in $\Sonetwo$.
By applying appropriate elements of $\Soneone$ to $(\alpha,\bv)$ we can obtain $(\alpha,\bu)$.
\end{proof}

The corollary implies that every starting board in the Spinpossible game has a solution.
We now give an upper bound on the length of any solution.

\begin{theorem}\label{thm:upper}
Every element of $\spinmn$ can be expressed as a product of at most $3mn - (m+n)$ spins.
\end{theorem}
\begin{proof}
Let $(\alpha,\bu)$ be an element of $\spinmn$.
For any $\bv\in V_N$ we may write $(\alpha,\bu)$ as $(\alpha,\bv)(\sid,\bu+\bv)$, where $\sid$ denotes the trivial permutation.
It is clear that we can express $(\sid,\bu+\bv)$ as the product of at most $mn$ elements in $\Soneone$.
Thus it suffices to show that we can construct an element of the form $(\alpha,\bv)$, for some $\bv \in V_N$, as a product of at most $2mn-(m+n)$ spins.
Since we may use any $\bv$ we like, we now ignore the orientation of tiles and focus on the permutation $\alpha$.
Rather than constructing $\alpha$, we shall construct $\alpha^{-1}$ (which is equivalent, since $\alpha$ is arbitrary).

We now proceed by induction on $N$ to show that we can construct $\alpha^{-1}$ using at most $2mn-(m+n)$ spins.
For $N=1$ we necessarily have $\alpha^{-1}=\sid$, which is the product of $0=2mn-(m+n)$ spins.
For $N>1$, assume without loss of generality that $m\le n$ (interchange the role of rows and colunmns in what follows if not).
We first use the spin about the rectangle $[1,\alpha(1)]$ to restore tile~1 to its correct position in the upper left corner.
Now let $i$ and $j$ be the vertical and horizontal distances, respectively, between positions $n+1$ and $\alpha(n+1)$, so that $\dist(n+1,\alpha(n+1))=i+j$.
To move tile $n+1$ to position $n+1$ (the second row of the leftmost column) we first apply an element of $\SS_{(i+1) \times 1}$ to move tile $n+1$ to the correct row, and then apply an element of $\SS_{1\times (j+1)}$ to move tile $n+1$ to the correct column (we can omit spins in $\Soneone$, which arise when $i$ or $j$ is zero).
Neither of these spins affects position~1.
In a similar fashion, we can successively move each tile $kn+1$ for $2\le k < m$ from position $\alpha(kn+1)$ to position $kn+1$ using at most two spins per tile, without disturbing any of the tiles in positions $jn+1$ for $j<k$.

The total number of spins used to correctly position all the tiles in the leftmost column is $2m-1$.
By the inductive hypothesis, we can correctly position the remaining tiles in the $(m-1)\times n$ board obtained by ignoring the leftmost column using at most $2(m-1)n-(m-1+n)$ spins.
The total number of spins used is
\[
2m-1 + 2(m-1)n-(m-1+n) = 2mn+m-3n-1,
\]
and since $m\le n$ this is less than $2mn-(m+n)$.
\end{proof}

The upper bound in Theorem \ref{thm:upper} can be improved for $mn>1$.
A more detailed analysis of the case $\spin_{3\times 3}$ shows that one can move every tile to its correct position using at most 9 spins\footnote{It is known that 7 spins always suffice, and are sometimes necessary.}, and then orient every tile correctly using at most 7 spins, yielding an upper bound of 16, versus the bound of 21 given by Theorem \ref{thm:upper}.
We also note that the leading constant 3 is not the best possible: for $m,n\ge 3$ the technique used to orient tiles in the $3\times 3$ case can be generalized to achieve $25/9$.

Let $k(m,n)$ denote the maximum length of a solution to a board in $\spinmn$.
Theorem \ref{thm:upper} gives an upper bound on $k(m,n)$.  We now prove a lower bound.

\begin{theorem}\label{thm:lower}
Assume $N=mn>1$.  Then
\[
k(m,n) \ge \frac{\ln (2^NN!)}{\ln \left(\binom{m+1}{2}\binom{n+1}{2}+1\right)}.
\]
This implies the bound
\[
k(m,n) \ge \frac{1}{2}mn - \frac{(1-\ln 2)}{2}\cdot\frac{mn}{\ln mn} + \frac{1}{4}.
\]

\end{theorem}
\begin{proof}
Let $c = |\SS|+1$.  The number of distinct expressions of the form $s_1\cdots s_j$ with $s_1,\ldots,s_j\in\SS$ and $j\le k$ is at most $c^k$.
Not all of these expressions yield distinct elements of $\spinmn$, but in any case it is clear that they correspond to at most $c^k$ distinct elements of $\spinmn$.
The cardinality of $\spinmn$ is $2^NN!$, thus in order to express every element of $\spinmn$ as a product of at most $k$ spins we must have
\begin{equation}\label{eq:lower1}
c^k \ge 2^NN!
\end{equation}
From Proposition \ref{prop:cardinality} we have $c = \binom{m+1}{2}\binom{n+1}{2} + 1$.
Taking logarithms in (\ref{eq:lower1}) and dividing by $\ln c$ yields the first bound for $k(m,n)$.

For the second bound, we note that $N^2 > c$ for all $N=mn>1$, thus we can replace the LHS of (\ref{eq:lower1}) by $N^{2k}$.
By bounding the error term in Stirling's approximation one can show that
\[
\ln N! \ge N\ln N - N + \frac{1}{2}\ln N,
\]
for all $N\ge 1$. 
Applying $N^{2k}> c^k$ and taking logarithms in (\ref{eq:lower1}) yields
\[
2k\ln N \ge  N\ln N - (1-\ln 2) N + \frac{1}{2}\ln N.
\]
Dividing by $2\ln N$ gives
\[
k \ge  \frac{1}{2}N - \frac{1-\ln 2}{2}\cdot\frac{N}{\ln N} + \frac{1}{4},
\]
which proves the second bound.
\end{proof}

For $m=n=3$, Theorem \ref{thm:lower} give the lower bound $k(3,3)\ge 6$, which is not far below the known value $k(3,3)=9$.
Asymptotically, we have the following corollary.
\begin{corollary}\label{cor:bounds}
The asymptotic growth of $k(m,n)$ is linear in $N=mn$.  More precisely, for every $\epsilon > 0$ there is an $N_0$ such that
\[
\left(\frac{1}{2}+\epsilon\right)N\medspace < \medspace k(m,n) \medspace \le \medspace 3N
\]
for all $N > N_0$.
\end{corollary}

Recall that for a group $G$ generated by a set $S$, the Cayley graph $\cay(G,S)$ is the graph with vertex set $G$ and edge $(g,h)$ labelled by $s$ whenever $sg=h$, where $s\in S$ and $g,h\in G$.
A solution to a board $b\in\spinmn$ corresponds to a shortest path from $b$ to the identity in the graph $\cay(\spinmn,\SS)$.  The quantity $k(m,n)$ is the diameter of this graph.

\section{Restricted Spin Sets}\label{section:restrictions}

Spinpossible includes variations of the standard game that place restrictions on the types of spins that are allowed.
For example, the ``no singles/doubles" puzzle levels prohibit the use of spins in $\Soneone$ and $\Sonetwo$.
This raises the question of whether it is still possible to solve every board under such a restriction.
More generally, we may ask: which subsets of the full set of spins $\SS=\SS(m,n)$ generate $\spinmn$?

We begin by defining three subsets of $\SS$ that cannot generate $\spinmn$ when $mn>1$, using three different notions of parity.

\begin{enumerate}
\item The $\emph{even area}$ spins $\SS^a$ are the spins whose rectangles have even area.
$\SS^a$ is the union of the $\SS_{i\times j}$ for which $ij\equiv 0\bmod 2$.
\smallskip

\item The $\emph{even permutation}$ spins $\SS^p$ are the spins that contain an even number of transpositions.
$\SS^p$ is the union of the $\SS_{i\times j}$ for which $ij\equiv 0\text{ or }1\bmod 4$.
\smallskip

\item The $\emph{even distance}$ spins $\SS^d$ are the spins that transpose positions at even distances.
$\SS^d$ is the union of the $\SS_{i\times j}$ for which $i+j\equiv 0\bmod 2$.
\end{enumerate}

We now consider the corresponding subgroups of $\spinmn$.  In these definitions $\alpha$ is a permutation in $S_N$,
$\bv$ is a vector in $V_N$, and $\wt(\bv)$ denotes the Hamming weight of~$\bv$ (the number of 1s it contains).
The group $A_N$ is the alternating group in~$S_N$, and we define the permutation group $D_N\cong S_{\lceil N/2\rceil}\times S_{\lfloor N/2 \rfloor}$ as follows:
\[
D_N = \{ \alpha : \dist(i,\alpha(i)) \equiv 0\bmod 2 \text{ for } 1\le i\le N\}.
\]
Here $i$ and $\alpha(i)$ identify positions on an $m\times n$ board and $\dist(i,\alpha(i))$ is the $\ell_1$-distance.

The subgroups $\spinmn^*$, where $*$ is $a$, $p$, or $d$, are defined as follows:
\begin{enumerate}
\item $\spinmn^a = \{(\alpha,\bv):\wt(\bv)\equiv 0 \bmod 2\}\qquad\text{(index 2)}$.
\smallskip
\item $\spinmn^p = \{(\alpha,\bv): \alpha\in A_N\}\qquad\qquad\qquad\text{(index 2)}$.
\smallskip
\item $\spinmn^d = \{(\alpha,\bv): \alpha\in D_N\}\qquad\qquad\qquad\text{(index }\binom{N}{\lfloor N/2 \rfloor})$.
\smallskip
\end{enumerate}
It is not necessarily the case that $\SS^*$ generates $\spinmn^*$, but we always have $\SS^* = \SS \cap \spinmn^*$.
In particular, it is clear that $\langle\SS^*\rangle\subseteq\spinmn^*$.
The following propositions give some conditions under which equality holds.

\begin{proposition}\label{prop:spina}
Assume $m,n\ge 2$ and $mn>4$.
Then $\Sonetwo\cup\Stwotwo$ (and therefore~$\SS^a$) generates~$\spinmn^a$.
\end{proposition}
\begin{proof}
Let $G=\langle \Sonetwo\cup\Stwotwo\rangle$, and let $\hat\projmap$ denote the restriction of $\projmap$ to $G$.
The fact that $G$ contains $\Sonetwo$ implies that $\projmap(G) = S_N$, by Lemma \ref{lemma:transpositions}.
It thus suffices to show that the kernel of $\hat\projmap$ has index 2 in $\ker \projmap\cong V_N$.

The following product of spins in $G$ transposes the tiles in positions 1 and 2:
\begin{equation}\label{eq:agen}
[2,3] [1,n+1] [1,n+2] [2,3] [1,n+2] [1,n+1] [2,3] = \bigl((1\ 2),\bid\bigr)
\end{equation}
We can transform the identity above by applying any square-preserving isometry of~$\Z^2$ (the group generated by unit translations and reflections about the lines $y=0$ and $y=x$).
Such a transformation may change the location and/or orientation of the rectangles identifying the spins that appear in the product, but it does not change their spin type (the set $\SS_{i\times j}$ to which they belong).
This allows us to transpose any pair of adjacent tiles on the $m\times n$ board using a product of spins in~$G$.
It follows from Lemma \ref{lemma:transpositions} that $G$ contains the subgroup $H=\{(\alpha,\bid):\alpha\in S_N\}$.

For each even integer $w$ from $0$ to $N$, we can construct some $g_w=(\beta,\bv)$ with $\wt(\bv) = w$, as a product of elements in $\Sonetwo$.
The coset $g_wH\subset G$ then contains elements of the form $(\beta,\bv)$ for every vector $\bv$ with $\wt(\bv) = w$.
Multiplying each $(\beta,\bv)$ on the left by $(\beta^{-1},\bid)$,
we see that $G$ contains elements $(\sid,\bv)$ for every even weight vector $\bv$.
Therefore $\ker \hat\projmap$ has index 2 in $\ker \projmap$.
\end{proof}

We note that $\SS^a$ does not generate $\spinmn^a$ when $m=n=2$, nor when exactly one of $m$ or $n$ is 1.

\begin{proposition}\label{prop:spind}
Assume $mn \ne 4$.  Then $\Soneone\cup\Stwotwo\cup\Sonethree$ (and therefore $\SS^d$) generates $\spinmn^d$.
\end{proposition}
\begin{proof}
Let $G=\langle \Soneone\cup\Stwotwo\cup\Sonethree\rangle$.  Then $G$ contains $\ker \projmap= \langle\Soneone\rangle\cong V_N$.
It remains to show that $\projmap(G)=D_N$.

Assume for the moment that $m\ge 2$ and $n\ge 3$.
The following product of spins in $G$ transpose the tiles in positions 1 and n+2:
\begin{equation}\label{eq:dgen}
[2,n+3] [1,3] [2,n+3] [n+3,n+3] = \bigl((1\ n+2),\bid\bigr)
\end{equation}
As in the proof of Proposition \ref{prop:spina}, we may transform this identity by applying any square-preserving isometry of $\Z^2$.
Thus we can transpose any pair of tiles that share exactly one common vertex (i.e., that are ``diagonally adjacent"),
and we can also handle the case $m\ge 3$ and $n\ge 2$.
It then follows from Lemma \ref{lemma:transpositions} that these transpositions generate $D_N$.

We now consider the case where $m$ or $n$ is 1.
If $mn\le 2$ the proposition clearly holds (we only need spins in $\Soneone$), so assume without loss of generality that $m=1$ and $n\ge 3$.
We now replace (\ref{eq:dgen}) with
\[
[1,1] [1,3] [2,2] [1,1] = \bigl((1\ 3),\bid\bigr),
\]
and apply the same argument.
\end{proof}
\noindent
It is easy to check that when $mn=4$ the set $\SS^d$ does not generate $\spinmn^d$.

We leave open the question of when $\SS^p$ generates $\spinmn^p$, but for $\spin_{3\times 3}$ we note that $\SS^p \subset \SS^a$ (see below), thus $\SS^p$ does not generate $\spinmn$ in this case.

For reference, we list the spin types $\SS_{i\times j}$ contained in $\SS^a$, $\SS^p$, and $\SS^d$ for all $i,j\le 3 \le m,n$:

\begin{enumerate}
\item $\Sonetwo, \Stwotwo, \Stwothree \subset \SS^a$.
\smallskip
\item $\Stwotwo, \Sthreethree \subset \SS^p$.
\smallskip
\item $\Soneone, \Sonethree, \Stwotwo, \Sthreethree \subset \SS^d$.
\end{enumerate}
To simplify our analysis of the subsets of $\SS$ that generate $\spinmn$, we introduce an equivalence relation on spin types.
\begin{definition}
Two spin types $\SS_{i\times j}$ and $\SS_{i'\times j'}$ are equivalent, denoted $\SS_{i\times j} \sim \SS_{i'\times j'}$, whenever $\langle \SS_{i\times j}\rangle = \langle \SS_{i'\times j'}\rangle$.
\end{definition}

\begin{proposition}\label{prop:equivalence}
Let $m,n\ge 3$.  For $1\le i, i',j,j'\le 3$ there is exactly one non-trivial equivalence of spin types $\Sij \sim \SS_{i'\times j'}$,
namely, $\Sonetwo\sim \Stwothree$
\end{proposition}
\begin{proof}
If a particular spin type is contained in $\SS^*$ (where $*$ is $a$, $p$, or $d$), then so is every equivalent spin type.
Examining the list of spin types for $\SS^*$, we can use this criterion to rule out all but two possible equivalences among the 6 spin types $\SS_{i\times j}$ with $1\le i,j\le 3$.
The first is the pair $\Soneone$ and $\Sthreethree$, but these cannot be equivalent because $\langle\Soneone\rangle$ lies in $\ker\pi_1\cong V_N$ but $\Sthreethree$ does not.
The second is the pair $\Sonetwo$ and $\Stwothree$, which we now show are equivalent.

For simplicity we shall write spins in terms of rectangles with coordinates on a $3\times 3$ board, but these can be generalized to an $m\times n$ board by replacing positions 4, 5, 6, 7, 8 and 9
with positions $n+1$, $n+2$, $n+3$, $2n+1$, $2n+2$, and $2n+3$, respectively.
We can write the spin $[1,6]$ as a product of spins in $\Sonetwo$ as follows:
\[
[1,6] = [2,5] [2,3] [4,5] [5,6] [1,2] [4,5] [2,3] [3,6] [1,4].
\]
As in the proofs of Propositions \ref{prop:spina} and \ref{prop:spind}, we can transform this identity via a square-preserving isometry of $\Z^2$ to express any spin in $\Stwothree$ as a product of spins in $\Sonetwo$.  Thus $\langle \Stwothree\rangle \subset \langle \Sonetwo\rangle$.
For the other inclusion, we may write the spins $[1,2]$ and $[4,5]$ as products of spins in $\Stwothree$ as follows:
\begin{align*}
[1,2] &= [1,6] [4,9] [1,8] [4,9] [1,8] [1,6] [1,8] [4,9] [1,8] [1,6] [1,8] [4,9] [1,8] [1,6] [1,8],\\
[4,5] &= [1,6] [2,9] [1,6] [2,9] [4,9] [2,9] [1,6] [2,9] [4,9] [2,9] [1,6] [2,9] [4,9] [2,9] [4,9]
\end{align*}
By transforming one of these two identities with a suitable isometry we can express any spin in $\Sonetwo$ as a product of spins in $\Stwothree$.
Thus $\langle \Sonetwo\rangle \subset \langle \Stwothree\rangle$.
\end{proof}

We are now ready to prove our main theorem, which completely determines the combinations of spin types that generate $\spin_{3\times 3}$.

\begin{theorem}
Assume that $m,n\ge 3$.
Let $\TT$ be a union of spin types $\Sij$,  where $1\le i,j\le 3$.
For $\TT$ to generate $\spinmn$, it is sufficient for $\TT$ to contain one of the following six sets:
$$ \Sonetwo\cup\Soneone,\quad  \Sonetwo\cup\Sonethree \qquad \Sonetwo\cup\Stwotwo\cup\Sthreethree,$$
$$ \Stwothree\cup\Soneone,\quad  \Stwothree\cup\Sonethree \qquad \Stwothree\cup\Stwotwo\cup\Sthreethree.$$
When $m=n=3$, this condition is also necessary.
\end{theorem}
\begin{proof}
We first prove sufficiency.
By Proposition \ref{prop:equivalence}, $\Sonetwo\sim\Stwothree$, so it is enough to prove that each of the first three sets listed in the theorem generates $\spinmn$.
As above, we specify spins using coordinates on a $3\times 3$ board, but these can coordinates can be transferred to an $m\times n$ board as noted in the proof of Proposition \ref{prop:equivalence}.

By Corollary \ref{cor:generates}, the set $\Soneone\cup\Sonetwo$ generates $\spinmn$.
For $\Sonetwo\cup\Sonethree$, it is enough to show that $\Soneone\subset\langle\Sonetwo\cup\Sonethree\rangle$.
We note that each element of $\Soneone$ has the form $s_i=(\sid,\be_i)$, where $\be_i$ is the weight 1 vector in $V_N$ with the $i$th bit set.
If $(\alpha,\bu)$ is any element of $\spinmn$ with $\alpha(1)=i$, then we have
\begin{align*}
(\alpha,\bu)^{-1}(\sid,\be_1)(\alpha,\bu)&= (\alpha^{-1},\bu^{\alpha^{-1}})(\sid,\be_1)(\alpha,\bu)\\
                                         &= (\alpha^{-1},\bu^{\alpha^{-1}}+\be_1)(\alpha,\bu)\\
                                         &= (\sid,\be_1^\alpha) = (\sid,\be_i).
\end{align*}
Since $\projmap(\langle\Sonetwo\rangle)=S_N$, by Lemma \ref{lemma:transpositions}, we can generate a suitable $(\alpha,\bu)$ for each $i$ from 1 to $N$.
Thus it is enough to show how to express the spin $[1,1]$ as a product of spins in $\Sonetwo\cup\Sonethree$:
\[
[1,1] = [1,2] [1,4] [1,3] [1,4] [1,2] [4,6] [3,6] [4,6].
\]
The same arguments apply to the third set $\Sonetwo\cup\Stwotwo\cup\Sthreethree$, thus it suffices to note that:
\[
[1,1] = [2,3] [1,5] [1,2] [3,6] [4,7] [1,5] [2,3] [5,8] [1,9] [5,9] [1,5] [5,8] [8,9].
\]

We now prove the necessity of the condition in the theorem, under the assumption $m=n=3$.
The set $\TT$ is the union of some subset of the six spin types
\[
\UU = \bigl\{ \Soneone, \Sonetwo, \Sonethree, \Stwotwo, \Stwothree, \Sthreethree \bigr\}.
\]
Of the 64 subsets of $\UU$, one finds that 22 of them have unions that are contained in $\SS^a$ or $\SS^d$, thus $\TT$ cannot
be the union of any of these 22 subsets.
Conversely, one finds that 39 of the remaining 42 subsets of $\UU$ have unions containing one of the 6 sets listed in the proposition.
The 3 remaining subsets of $\UU$ all have unions contained in $\Sonetwo\cup\Stwothree\cup\Sthreethree$, which we now argue does not generate $\spinmn$.

Since $\Sonetwo\sim\Stwothree$, it is enough to show that $\Sonetwo\cup\Sthreethree$ does not generate $\spinmn$.
By Lemma \ref{lemma:unique} below, any product of elements in $\Sonetwo\cup\Sthreethree$ is equivalent to a product in which the unique element of $\Sthreethree$ appears only once, in the rightmost position.
It follows that the cardinality of $\langle\Sonetwo\cup\Sthreethree\rangle$ is at most (in fact, exactly) twice that of $\langle\Sonetwo\rangle$.
But by Lemma \ref{lemma:onetwo} below, the subgroup $\langle\Sonetwo\rangle$ has trivial interesection with $\ker\projmap$ and thus has index $2^9=512$ in $\spin_{3\times3}$.
So $\langle\Sonetwo\cup\Sthreethree\rangle$ is a proper subgroup of $\spin_{3\times 3}$.
\end{proof}

\begin{lemma}\label{lemma:onetwo}
The restriction of the projection map $\projmap\colon\spinmn\to S_N$ to the group $G=\langle \Sonetwo\rangle$ is an isomorphism from $G$ to $S_N$.
\end{lemma}
\begin{proof}
Let $\hat\projmap\colon G\to S_N$ be the restriction of $\projmap$ to $G$.
The fact that $\hat\projmap$ is surjective follows from Lemma \ref{lemma:transpositions}, so we only need to show that $\hat\projmap$ is injective.
Let $h$ be any element of the kernel of $\hat\projmap$.
Then $h = s_1\cdots s_k$ is a product of spins in $\Sonetwo$, and $h$ fixes the position of every tile on the $m\times n$ board.
We will show that $h$ also fixes the orientation of every tile, and therefore $h$ is the identity.

Consider tile $t$ in position $t$ on the standard board $b$.
If we apply $h$ to $b$, each spin $s_i$ potentially moves the tile $t$, but if it does, it always moves $t$ to an adjacent position on the board, since $s_i\in\Sonetwo$.
Thus $t$ is moved along some path on the $m\times n$ board (possibly trivial) that must eventually return $t$ to its original position.
The length of this path is necessarily an even integer, therefore $t$ is also returned to its original orientation.
\end{proof}

Let $\TT$ be a subset of the spins in $\SS$.
Generalizing our definition of $k(m,n)$, we define $k(m,n,\TT)$ as the diameter of the Cayley graph $\cay(\spinmn,\TT)$, and consider upper and lower bounds for $k(m,n,\TT)$.
To do so, we introduce a notion of \emph{weight} for a spin, defined the total distance coverd by all the tiles it moves.
\begin{definition}
The \emph{weight} of a rectangle $R$ is $\wt(R) = 2\sum_{p\in R}\rho(p,R)$,
and the weight of a spin~$s$ about $R$ is $\wt(s) = \wt(R)$.
\end{definition}
We may denote the weight of an $i\times j$ rectangle $R$ by $w(i,j)$, since it depends only on the dimensions of $R$, not its location.

\begin{lemma}
Let $\varepsilon:\Z\to \{0,1\}$ be the parity map.  Then
\[
w(m,n) = \frac{1}{2}\bigl(mn^2 + nm^2 - \varepsilon(m)n - \varepsilon(n)m\bigr).
\]
\end{lemma}
\begin{proof}
When $m$ and $n$ are both even we have
\[
w(m,n) = 4\cdot 2\left(\sum_{i=1}^{\nicefrac{m}{2}}\sum_{j=1}^{\nicefrac{n}{2}}(i+j-1)\right) = \frac{1}{2}\bigl(mn^2+ nm^2\bigr).
\]
When $m$ and $n$ are both odd we have
\begin{align*}
w(m,n) &= 4\cdot 2\left(\sum_{i=1}^{\frac{m-1}{2}}\sum_{j=1}^{\frac{n-1}{2}}(i+j)\right) + 2\cdot 2\left(\sum_{i=1}^{\frac{m-1}{2}}i + \sum_{j=1}^{\frac{n-1}{2}} j\right)\\
       &= \frac{1}{2}\bigl(mn^2+nm^2 - m - n)\bigr).
\end{align*}
The cases where $m$ and $n$ have opposite parity are similar and left to the reader.
\end{proof}
\begin{lemma}\label{lemma:Tlower}
Let $\TT$ be any set of spins in $\spinmn$.  Then
\[
k(m,n,\TT)\medspace \ge\medspace \frac{w(m,n)}{\max\{\wt(s):s\in\TT\}}
\]
for all $m,n\ge 1$.
\end{lemma}
\begin{proof}
Let $b\in \SS_{m\times n}$, with $\wt(b)=w(m,n)$.  If $s_1\cdots s_k$ is a product of spins in $\TT$ equivalent to $b$, then $w(m,n)\le\sum \wt (s_i) \le kw_{\rm max}$.  The lemma follows.
\end{proof}

\begin{lemma}\label{lemma:Tupper}
Every element of $\spinmn$ can be expressed as the product of at most $w(m,n)+mn$ spins in $\Soneone\cup\Sonetwo$.
\end{lemma}
\begin{proof}
Let $b\in\spinmn$.  We will construct $b^{-1}=(\alpha,\bu)$ by constructing an element $(\alpha,\bv)$ as a product of at most $w(m,n)$ spins in $\Sonetwo$, to which we may then apply at most $mn$ spins in $\Soneone$ to obtain $(\alpha,\bv)$.

Let $d=m+n-2$.  Then $d$ is the maximum ($\ell_1$) distance between any position and the center of the $m\times n$ rectangle $R$ containing all the positions on the board.
For each position $i$ at distance $d$ from the center (the 4 corners when $mn>1$), we can move tile $i$ to position $i$ using at most $2d$ spins in $\Sonetwo$.
Next we place the correct tiles in positions at distance $d-1$ from the center, and each of these tiles can currently lie at most $d-1$ positions away from the center (since the distance $d$ positions are already filled with the correct tiles), thus we use at most $2(d-1)$ spins in $\Sonetwo$ to place the correct tiles in the positions at distance $d-1$ from the center.
Note that we can do this by moving each tile along a path that does not disturb any tiles that have already been placed.
Continuing in this fashion, we use at most $2\rho(p,R)$ spins to place the correct tile in position $p$, and the total number of spins is at most $\wt(R) = w(m,n)$.
\end{proof}

\begin{corollary}\label{cor:superlinear}
For all $m,n\ge 1$ let $\TT=\TT(m,n)$ be a set of spins with weight bounded by some constant $W$.
Than as $N=mn\to\infty$ we have the asymptotic bound $k(m,n,\TT)= \Theta(mn^2+nm^2)$.
More precisely, for every $\epsilon>0$ there is an $N_0$ such that
\[
\left(\frac{1}{2W}+\epsilon\right)(mn^2+nm^2)\medspace <\medspace k(m,n,\TT)\medspace <\medspace (1+\epsilon)(mn^2+nm^2),
\]
for all $N > N_0$.
\end{corollary}

For $m=n$ this gives a $\Theta(N^{1.5})$ bound, which may be contrasted with the $\Theta(N)$ bound of Corollary \ref{cor:bounds}, where the weight of the spins was unrestricted.
We note that in the case of $\spin_{3\times 3}$ and $\TT=\Soneone\cup\Sonetwo$, Lemmas \ref{lemma:Tlower} and \ref{lemma:Tupper} give the bounds $12 < k(3,3,\TT) < 33$, compared to the actual value $k(3,3,\TT)=25$.

\section{Unique Solutions}\label{section:unique}

Certain elements of $\spinmn$ are distinguished by the fact that they have a unique solution (a unique shortest expression as a product of spins).
This is clearly the case, for example, when $b\in\SS$.
There are many less trivial examples, some 2,203,401 of them in $\spin_{3\times 3}$. 
These include what appear to be the most difficult puzzles in the game, some of which are featured in separate puzzle levels designated as ``uniques".
While these can be quite challenging, knowing that the solution is unique can be an aid to solving such a puzzle.

We begin with a lemma used in the proof of Theorem 3, which also allows us to rule out many possible candidates for a unique solution.

\begin{lemma}\label{lemma:unique}
Let $b=s_1\cdots s_i\cdots s_k$ be a product of spins in $\spinmn$, with $i<k$ and $s_i\in\Soneone\cup\SS_{m\times n}$.
If $s_i\in\Soneone$, then $b$ can be written as $b=s_1\cdots s_{i-1}s_{i+1}\cdots s_k t_i$ with $t_i\in \Soneone$.
If $s_i\in\SS_{m\times n}$, then $b$ can be written as $b=s_i\cdots s_{i-1}t_{i+1}\cdots t_k s_i$, with each $t_j$ a spin of the same type as $s_j$, for $i\le j\le k$.
\end{lemma}
\begin{proof}
We first suppose that $s_i\in \Soneone$.  Then the rectangle $R_i$ of $s_i$ contains just a single position.
Let $R_{i+1}$ be the rectangle of $s_{i+1}$.
If $R_i$ is contained in $R_{i+1}$, then by Proposition \ref{prop:properties}, we have $s_{i+1}s_is_{i+1}=t$ with $t\in\Soneone$.
Multiplying on the left by $s_{i+1}$, we have $s_is_{i+1}=s_{i+1}t_i$, allowing us to ``shift" the spin $s_i$ to the right, potentially changing the location of its rectangle but not its type.
If $R_i$ is not contained in $R_{i+1}$ then $R_i$ and $R_{i+1}$ are disjoint and we simply let $t=s_i$, since then $t$ and $s_{i+1}$ commute.
Applying the same procedure to $s_{i+2},\ldots,s_k$, we eventually obtain a product $b=s_1\cdots s_{i-1}s_{i+1}\cdots s_k t_i$ of the desired form (using a potentially different $t$ at each step).

We now suppose that $s_i\in \SS_{m\times n}$.  Then the rectangle $R_i$ of $s_i$ covers the entire $m\times n$ board.
Let $R_{i+1}$ be the rectangle of $s_{i+1}$, which is necessarily contained in~$R_i$.
We then have $s_is_{i+1}s_i=t_{i+1}$, where $t_{i+1}$ is a spin of the same type as $s_{i+1}$, and therefore $s_is_{i+1} = t_{i+1}s_i$.
We may proceed in the same fashion to compute $t_{i+2},\ldots,t_k$, eventually obtaining the desired product $b=s_i\cdots s_{i-1}t_{i+1}\cdots t_k s_i$.
\end{proof}

\begin{proposition}
Suppose that $s_1\cdots s_k$ is the unique solution to a board $b$ in $\spinmn$.
Then the following hold:
\begin{enumerate}
\item None of the $s_i$ are contained in $\Soneone$ or $\SS_{m\times n}$.
\item Consecutive pairs $s_i$ and $s_{i+1}$ have rectangles $R_i$ and $R_{i+1}$ that overlap and do not share a common center, with neither contained in the other.
\end{enumerate}
\end{proposition}
\begin{proof}
(1) follows from Lemma \ref{lemma:unique} and its proof: if $s_i$ were an element of $\Soneone$ or $\SS_{m\times n}$ we could obtain a different expression for $b$ as a product of spins of the same length by ``shifting" $s_i$ either to the left or right.
(2) follows from parts (3) and (4) of Proposition \ref{prop:properties}.
\end{proof}

We conclude with a list of some open problems:
\renewcommand\labelenumi{\theenumi.}
\begin{enumerate}
\item Give a short proof that $k(3,3)=9$.
\item Determine $k(4,4)$.
\item Determine whether $\lim_{n\to\infty} k(n,n) / n^2$ exists, and if so, its value.
\item Analyze the distribution of solution lengths in $\spin_{m\times n}$.
\item Determine which spin types $\Sij$ are equivalent.
\item Determine which combinations of spin types generate $\spin_{4\times 4}$.
\item Give bounds on the number of boards with unique solutions in $\spin_{m\times n}$.
\end{enumerate}


\end{document}